\documentclass[11pt]{amsart}
\usepackage{geometry}
\usepackage{microtype}
\numberwithin{equation}{section}
\usepackage{amssymb}
\usepackage{xcolor}
\usepackage[T1,T5]{fontenc}
\usepackage{hyperref}
\hypersetup{   
colorlinks=true,       
linkcolor=black,          
citecolor=black,        
filecolor=black,      
urlcolor=black}     

\newtheorem{theorem}{Theorem}[section]
\newtheorem{lemma}[theorem]{Lemma}

\newtheorem{prop}[theorem]{Proposition}
\newtheorem{proposition}[theorem]{Proposition}

\newtheorem{coro}[theorem]{Corollary}

\theoremstyle{definition}
\newtheorem{defi}[theorem]{Definition}

\newtheorem{example}[theorem]{Example}
\newtheorem{remark}[theorem]{Remark}
 \newtheorem*{ackn}{Acknowledgements}

 \newcommand{\R}{\mathbb R}

 \newcommand{\N}{\mathbb N}

\begin{document}

\title[Complex Monge-Amp\`ere equations on compact Hermitian manifolds]{Complex Monge-Amp\`ere equations on Hermitian manifolds:  from bounded to smooth solutions}

  \author{Papa Badiane}
\address{Papa Badiane: Classes Préparatoires aux Grandes \'Ecoles (CPGE) de Thiès, BP 36005, Dakar (Sénégal)\and
	Laboratoire de Math\'ematiques et Applications\\ UFR des Sciences et Technologies\\
Université Assane Seck de Ziguinchor, BP 523 (S\'en\'egal) }

\email{\href{mailto:p.badiane4963@zig.univ.sn }{\href{mailto:p.badiane4963@zig.univ.sn}{p.badiane4963@zig.univ.sn}}}

\author{Chinh H. Lu}
\address{Chinh H. Lu: Institut Universitaire de France \& Univ Angers, CNRS, LAREMA, SFR MATHSTIC, F-49000 Angers, France}
\email{\href{hoangchinh.lu@univ-angers.fr}{hoangchinh.lu@univ-angers.fr}}

\author{Ahmed Zeriahi}
\address{Ahmed Zeriahi: Institut de Math\'ematiques de Toulouse   \\ Universit\'e de Toulouse \\
118 route de Narbonne \\
31400 Toulouse, France\\ }

\email{\href{mailto:ahmed.zeriahi@math.univ-toulouse.fr}{ahmed.zeriahi@math.univ-toulouse.fr}}

 \begin{abstract}
 	We study regularity of bounded weak solutions to complex Monge-Amp\`ere equations on compact Hermitian manifolds with right-hand side possibly decreasing in the unknown. Our proof relies on the domination principle and a priori estimates, partially motivated by the Monge-Amp\`ere eigenvalue problem. Our main result roughly says that any bounded solution is smooth in the regular locus of the data.  When the right-hand side is strictly positive and smooth, we recover known results by Nie, Ko{\l}odziej-Nguyen for the Hermitian case and Sz\'ekelyhidi-Tosatti for the K\"ahler case, which rely on the regularizing property of the K\"ahler-Ricci flow. 
 \end{abstract}

\date{\today}
 
 \maketitle
 
 \tableofcontents

\section{Introduction}

Since the breakthrough of Yau \cite{Yau78} solving the Calabi conjecture, complex Monge-Amp\`ere equations have occupied a central place in complex geometry and analysis of PDEs.  Apart from the continuity method employed by Yau, several other approaches (viscosity, variational, Perron envelope, etc) for solving the equations have been developed in recent years. Solutions constructed from these methods have a priori low regularities and a natural question is whether they are as regular as the data. This is the main question we would like to study in this article.    

Let $X$ be a compact Hermitian manifold of dimension $n$ and fix a Hermitian metric $\omega_X$ on $X$. Let $\theta$ be a smooth semipositive $(1,1)$-form that is big. By big, we mean there exists a function $\rho$ with analytic singularities along an analytic subset $D$ such that $\theta +dd^c \rho \geq 2\varepsilon_0 \omega_X$, for some $\varepsilon_0>0$ that will be fixed.

We consider the following complex Monge-Amp\`ere equation 
\begin{equation}\tag{CMAE}
	\label{eq: CMAeq}
	(\theta+dd^c u)^n = e^{f(\cdot,u) + \psi^+-\psi^{-}} \omega_X^n,
\end{equation}
where $f: X \times \mathbb R \rightarrow \mathbb R$ is a smooth function, and $\psi^{\pm}$ are quasi-psh functions on $X$ which are smooth in a fixed  open set $\Omega\subset X$ (not necessarily Zariski dense in $X$). The existence of a bounded $\theta$-psh function solving \eqref{eq: CMAeq} is not guaranteed, but our main result shows that, if it exists, it is smooth in the regular locus of the data.

\begin{theorem}\label{thm: bdd implies smooth intro}
	If $\varphi$ is a bounded $\theta$-psh function solving \eqref{eq: CMAeq}, then $\varphi$ is smooth in $\Omega\setminus D$.  
\end{theorem}
The set $\Omega$ is not necessarily open in the Zariski topology, and it can merely be a small open set in $X$. The existence of a bounded solution implicitly implies that the function $e^{\psi^+-\psi^-}$ is integrable over $X$. In case $f$ is non-decreasing in $t$, the latter integrability condition is also sufficient to ensure the solvability of \eqref{eq: CMAeq}, and as shown in \cite{DnL17} in the K\"ahler case with $\Omega$ being a Zariski open set, the unique solution is smooth in $\Omega\setminus D$. 

In a particular case when $\psi^{\pm}=0$ and $\theta$ is K\"ahler, the result was known by Sz\'ekelyhidi and Tosatti \cite{ST11} who proved that $\varphi$ is smooth by using the smoothing property of the weak K\"ahler-Ricci flow \cite{ST17},  \cite{Tsu88}, \cite{TZ06}. Nie \cite{Nie14} used the idea of \cite{ST11} to study the same problem on compact Hermitian manifolds, and obtained the same result under a condition on the preservation of mass which, as shown by Chiose \cite{Chio24}, turns out to be equivalent to the curvature assumption introduced by Guan and Li \cite{GL10}.  What was actually used in \cite{Nie14} is a stability result for solutions to complex Monge-Amp\`ere equations. This deep result was finally obtained by Ko{\l}odziej and Nguyen \cite{KN19}, hence the smoothness of bounded solutions is also true in the Hermitian case.

When $f$ is non-decreasing in $t$ and $\psi^{\pm}=0$, $\theta>0$, the existence and uniqueness of a smooth solution were proved by Tosatti and Weinkove \cite{TW10} building on earlier contributions by Cherrier \cite{Cher87}, Hanani \cite{Han96}, Guan and Li \cite{GL10}, among others. For more general right-hand sides, weak solutions have been investigated by various authors; see, for instance, \cite{DK12ALM}, \cite{KN19}, \cite{Ng16AIM}, \cite{KN15Phong}, and \cite{GL23Crelle}. 

When $f$ is not increasing in $t$, the situation becomes more delicate: solutions, when they exist, are generally not unique, except in certain special cases (see \cite{LP25}). 

 Our proof is different from that of \cite{ST11}, it is partially inspired by the a priori estimates appearing in the study of the Monge-Amp\`ere eigenvalue problem in \cite{BZ23}.  Let us explain the main idea in the simplest case, when $\theta=\omega$ is K\"ahler and $\psi^{\pm}=0$, so that $\Omega = X$ and $D = \emptyset$. We approximate the bounded solution $\varphi$ by a decreasing sequence $(v_j)$ of smooth $\omega$-psh functions via Demailly's approximation theorem \cite{Dem94}. For each $j$, we solve 
\[
(\omega+dd^c u_j)^n = e^{u_j-v_j + f(\cdot,v_j)} \omega_X^n.  
\]
We next establish uniform estimates for $u_j$ allowing to extract a subsequence which uniformly converges to a smooth function $u$ satisfying  
 \[
(\omega+dd^c u)^n = e^{u-\varphi + f(\cdot,\varphi)} \omega_X^n.  
\]
The domination principle then ensures that $u=\varphi$. Given this scheme, the main task is to establish uniform estimates independent of $v_j$. This approach is strongly inspired by the recent work of the first and last authors on the eigenvalue problem \cite{BZ23}.  

In the case $\psi^{\pm}=0$, our approach yields a simpler alternative proof to those of \cite{ST11} and \cite{Nie14}. Furthermore, our assumptions on the data are sufficiently flexible to extend the result to singular varieties, providing analogous regularity for the solutions.  
\begin{coro}
	Let $V$ be a compact complex variety with klt singularities equipped with an Hermitian form $\omega_V$, an adapted measure $\mu_V$, and let $f : V\times \mathbb R \rightarrow \mathbb R$ be a smooth function. Then any bounded solution to the Monge-Amp\`ere equation 
	\[
	(\omega_V+dd^c u)^n = e^{f(\cdot,u)}\mu_V, \; \omega_V+dd^c u\geq 0,
	\]
	is smooth in the regular locus $V_{{\rm reg}}$. 
\end{coro}

For more details on singular varieties and adapted measures, we refer the reader to \cite{EGZ09}, \cite{GL23Crelle}. The above corollary is a direct consequence of Theorem \ref{thm: bdd implies smooth intro} since working on a desingularization $\pi: X \rightarrow V$, the function $u\circ \pi$ is bounded, $\pi^*\omega_V$-psh and it solves \eqref{eq: CMAeq} for appropriate $\psi^{\pm}$.

In the case when $V$ is a $\mathbb Q$-Fano compact K\"ahler variety with klt singularities, our result is applicable showing that any singular K\"ahler-Einstein metric is smooth in the regular locus of $V$. It was announced  in \cite{GKSS25} (see the text below Definition 2 in \cite{GKSS25}) that the above Corollary can be proved by using the smoothing property of the K\"ahler-Ricci flow as in \cite{ST11}. 
	
	\medskip
	
	When the right-hand side in \eqref{eq: CMAeq} is not increasing in $u$, existence of solutions is not guaranteed. Guan-Li \cite{GL10} proved that on compact Hermitian manifolds with boundary the existence of smooth subsolutions ensures the solvability of the equation. On compact manifolds without boundary, the situation is different, and the existence of a subsolution is not enough in general, but we have the following. 
\begin{theorem} \label{thm:SubSolThm}
Assume that the equation \eqref{eq: CMAeq} admits a bounded subsolution $\underline u \in {\rm PSH}(X,\theta)$ and a bounded supersolution $\bar u  \in {\rm PSH}(X,\theta)$ such that $\underline{u}\leq \bar u$. Then it admits a bounded solution $u \in {\rm PSH}(X,\theta)$ which is smooth in $\Omega$ and satisfies $\underline u \leq u \leq \bar u$ on $X$.
\end{theorem}

As shown in Example \ref{example: non-increasing f}, bounded subsolutions and supersolutions always exist when $f(\cdot,t)$ is increasing for sufficiently large $|t|$. On the other hand, the existence of a subsolution and a supersolution (without any ordering between them) does not guarantee the existence of a solution, as illustrated in Example \ref{example:general f}.

When $f$ is non-decreasing in $t$, the situation becomes simpler. Indeed, Theorem \ref{thm: bdd sol f increasing} shows that the existence of a bounded subsolution and a bounded supersolution, not necessarily ordered, is sufficient to guarantee solvability of the equation.

\medskip 

 \noindent {\bf Organization of the article.} In the next section we recall elementary and useful inequalities used in the Laplacian estimate, and we also recall useful tools in pluripotential theory.  For pedagogical reason, we start proving Theorem \ref{thm: bdd implies smooth intro} in the simplest K\"ahler setting in Section \ref{sect: Kahler smooth}, and  we will take care of the Hermitian case in   Section \ref{sect: Herm case}. Section \ref{sect: subsol thm} is devoted to the proof of Theorem  \ref{thm:SubSolThm}.

 \begin{ackn} 
This work is partially supported by Institut Universitaire de France, the KRIS project of Fondation Charles Defforey, as well as the Centre Henri Lebesgue ANR-11-LABX-0020-01. There is no use of AI in our work. 
\end{ackn}

\section{Preliminaries}
In the whole article we let $(X,\omega_X)$ be a compact complex manifold of dimension $n$, and fix a Hermitian metric $\omega_X$ on $X$. We also fix a smooth real $(1,1)$-form $\theta$ that is semipositive and big. By big we mean there exists a function $\rho$ with analytic singularities along an analytic set $D\subset X$, such that $\theta+dd^c \rho \geq 2\varepsilon_0 \omega_X$, for some constant $\varepsilon_0>0$. We also assume that $\theta \leq \omega_X$. 

 \subsection{Elementary inequalities}  
 \begin{lemma}\label{lem: elementary inequalities}
Let $\alpha, \beta, \gamma$  be positive $(1,1)$-forms on $X$. Then
$$
\left(\frac{\beta^n}{\alpha^n}\right)^{\frac{1}{n}}\leq \dfrac{1}{n} {\rm Tr}_{\alpha}(\beta)\leq  \left(\frac{\beta^n}{\alpha^n}\right)\cdot \left( {\rm Tr}_{\beta}(\alpha)\right)^{n-1}, 
$$
and ${\rm Tr_{\alpha}(\beta)}\rm Tr_{\beta}(\gamma) \geq \rm Tr_{\alpha}(\gamma)$. 
\end{lemma}

Recall that ${\rm Tr}_{\alpha}(\beta):= n\dfrac{\alpha^{n-1}\wedge \beta}{\alpha^{n}}$.

\begin{proof}
	The trace of $\beta$ with respect to  $\alpha$ is ${\rm Tr}_{\alpha}(\beta) = \sum_{i}\lambda_{i}$ where $ 0< \lambda_{1} \leq \ldots \leq \lambda_{n}$ are the eigenvalues of $\beta$ with respect to $\alpha$.  The first two inequalities are equivalent to
\[
\left( \prod_{i} \lambda_{i} \right)^{\frac{1}{n}}\leq \dfrac{1}{n} \sum_{i}\lambda_{i}\leq  \left(\prod_{i} \lambda_{i} \right)\cdot \left( \sum_{i}\lambda_{i}^{-1} \right)^{n-1},
\]
which are true by elementary reason.

By a linear change of coordinates around a given point $x_0$, we can write 
\[
\beta= \sqrt{-1} \sum_{j,k} dz_j \wedge d\bar z_k, \; \alpha = \sqrt{-1}\sum_{j} \alpha_j dz_j \wedge d\bar z_j, \; \gamma = \sqrt{-1}\sum_{j,k} \gamma_{jk} dz_j \wedge d\bar z_k. 
\]
At this point $x_0$, the second inequality in the lemma then reads 
\[
\sum_{j,k} \alpha_{j}^{-1} \gamma_{k\bar k} \geq \sum_{j} \alpha_j^{-1} \gamma_{j\bar j},
\]
which is true because all terms involved are non negative. 
\end{proof}

\subsection{The Monge-Amp\`ere operator}
A function $u$ is quasi-plurisubharmonic (or quasi-psh for short) on $X$ if locally $u=g+\varphi$, where $g$ is smooth and $\varphi$ is psh. We say that $u$ is $\theta$-psh if it is quasi-psh and $\theta+dd^c u\geq 0$ in the sense of currents. We let ${\rm PSH}(X,\theta)$ be the set of $\theta$-psh functions that are integrable on $X$. By basic properties of psh functions the latter means exactly that $u$ is not identically $-\infty$. 

Given a bounded $\theta$-psh function $u$, the Monge-Amp\`ere operator $(\theta+dd^c u)^n$ is well-defined as a positive Radon measure on $X$ (see \cite{DK12ALM},\cite{GL22} where the Bedford-Taylor theory \cite{BT76,BT82} is adapted into the Hermitian setting) whose total mass is not zero. This latter fact can be proved using the modified comparison principle in \cite{KN15Phong}. An alternative proof using quasi-psh envelopes can be found in \cite{DnL25}. 

One of the most important and powerful tools in pluripotential theory is the following comparison principle. 
\begin{prop}\label{prop: max, comp, dom principle}
	Assume $u,v$ are bounded $\theta$-psh functions. Then 
	\[
	{\bf 1}_{\{u>v\}}(\theta+dd^c \max(u,v))^n ={\bf 1}_{\{u>v\}}(\theta+dd^c u)^n,
	\]
	\[
	{\bf 1}_{\{u=v\}}(\theta+dd^c \max(u,v))^n \geq {\bf 1}_{\{u=v\}}(\theta+dd^c u)^n.
	\]
	If ${\bf 1}_{\{u<v\}}(\theta+dd^c u)^n \leq c {\bf 1}_{\{u<v\}} (\theta+dd^c v)^n$, for some $0\leq c<1$, then $u\geq v$. In particular, 
	\begin{enumerate}
		\item if $(\theta+dd^c u)^n\leq e^{u-v}(\theta+dd^c v)^n$, then $u\geq v$; 
	\item if if $(\theta+dd^c u)^n\leq c(\theta+dd^c v)^n$, for some $c\geq 0$, then $c\geq 1$. 
	\end{enumerate}
\end{prop}

The first two (in)equalities are often referred to as the pluripotential maximum principle, while the last conclusion is called the domination principle. We refer the reader to \cite{GZbook}, for a proof in the K\"ahler case and \cite{GL23Crelle}, \cite[Corollary 2.27]{DnL25} for the Hermitian case.

Given a function $h$ we define $P_{\theta}(h)$ to be the largest $\theta$-psh functions lying below $h$. It is the upper semicontinuous regularization of the upper envelope of all $\theta$-psh functions lying below $h$. When $h$ is smooth and $\theta>0$, one can approximate $P_{\theta}(h)$ by using Berman's $\beta$-convergence method \cite{Ber19}, as illustrated in \cite{CZ19}. The consequence of this approximation scheme is that $P_{\theta}(h)$ is in $C^{1,1}(X)$ and the Monge-Amp\`ere measure $(\theta+dd^c P_{\theta}(h)^n$ is concentrated on the contact set $\{P_{\theta}(h)=h\}$. The latter fact is often referred to as the projection theorem. For less regular $h$ and for merely semipositive and big form $\theta$, the projection theorem still holds, as explained in \cite{GL22} building on \cite{GLZ19}. 

\begin{prop}
	Assume $h$ is quasi-continuous and bounded from below on $X$. Then $P_{\theta}(h)$ is a bounded $\theta$-psh function whose Monge-Amp\`ere measure $(\theta+dd^c P_{\theta}(h))^n$ vanishes in $\{P_{\theta}(h) <h\}$. 
\end{prop}

It was also proved in \cite{GLZ19} that the pluripotential minimum principle holds. Together with the pluripotential maximum principle, this gives a very flexible way to study limits of solutions to complex Monge-Amp\`ere equations, especially when convergence in capacity is missing.  
\begin{prop}\label{prop: min principle}
	Assume $\mu$ is a positive non-pluripolar measure on $X$, $g : X \times \R \rightarrow \mathbb R$ locally bounded function continuous in the second variable, $(f_j)$ is a sequence of non-negative $L^1(\mu)$-functions, and $(u_j)$ is a uniformly bounded sequence of $\theta$-psh functions such that 
	\[
	(\theta+dd^c u_j)^n \leq g(\cdot,u_j) f_j \mu. 
	\] 
	Then $u:=P_{\theta}(\inf_j u_j)$ is in ${\rm PSH}(X,\theta)\cap L^{\infty}$, and 
	\[
	(\theta+dd^c u)^n \leq g(\cdot,u) \left(\sup_j f_j\right)\mu. 
	\]
\end{prop}

\begin{prop} \label{prop: max principle}
	Assume $\mu$ is a positive non-pluripolar measure on $X$, $g : X \times \R \rightarrow \mathbb R$ locally bounded function continuous in the second variable, $(f_j)$ is a sequence of non-negative $L^1(\mu)$-functions, and $(u_j)$ is a uniformly bounded sequence of $\theta$-psh functions such that 
	\[
	(\theta+dd^c u_j)^n \geq g(\cdot,u_j) f_j \mu. 
	\] 
	Then $u:= (\sup_j u_j)^*$ is a bounded $\theta$-psh function and 
	\[
	(\theta+dd^c u)^n \geq g(\cdot,u) \left(\inf_j f_j\right)\mu. 
	\]
\end{prop}

The following trick was used several times in \cite{DDL5},  \cite{LZ25} and we will use it again in this article. 

\begin{lemma}\label{lem: b trick}
	Assume $u,\psi$ are bounded $\theta$-psh functions and set $u_b:=P_{\theta}(bu-(b-1)\psi)$, where $b>1$ is a constant. Then 
	\[
	(\theta+dd^c u_b)^n \leq {\bf 1}_{\{u_b=bu-(b-1)\psi\}}b^n (\theta+dd^c u)^n. 
	\]
\end{lemma}
\begin{proof}
	From the projection theorem we know that $(\theta+dd^c u_b)^n$ is concentrated on the contact set $E:=\{u_b=bu-(b-1)\psi\}$. Moreover, $b^{-1}u_b+(1-b^{-1})\psi \leq u$ with equality on $E$. The pluripotential maximum principle thus gives 
	\[
	{\bf 1}_E (\theta+dd^c (b^{-1}u_b+(1-b^{-1})\psi))^n \leq {\bf 1}_E (\theta+dd^c u)^n.
	\] 
	Expanding the left-hand side and using positivity of $\theta+dd^c \psi$, we arrive at the result. 
\end{proof}

The following result can also be proved similarly. 
\begin{lemma}\label{lem: envelope two forms}
	Assume $0\leq \theta\leq \theta'$ are semipositive and big. Assume $u'$ is a bounded $\theta'$-psh function and let $u:=P_{\theta}(u')$. Then 
	\[
	(\theta+dd^c u)^n \leq {\bf 1}_{\{u=u'\}}(\theta'+dd^c u')^n. 
	\]
\end{lemma}

 \section{The K\"ahler case}
 In this section we prove our main result in the K\"ahler context. We first establish the a priori estimates in the simplest case when $\theta=\omega$ is K\"ahler and $\psi^{\pm}=0$, so that $\Omega=X$ and $D=\emptyset$.  

\begin{prop}\label{prop: Kahler case}
	Assume that $v$ is a smooth $\omega$-psh function such that $-C_0\leq v\leq C_0$, and $g : X\times [-C_0,C_0] \rightarrow \mathbb{R}$ is a smooth function.  Assume that $\varphi$ is a smooth $\omega$-psh function solving the Monge-Ampère equation
 \begin{equation} \label{eq:MAKahler}
	(\omega+dd^c \varphi)^n = e^{\varphi+ g(x,v)}\omega_X^n. 
 \end{equation} 
	Then,
	\[
	\sup_X |\varphi| + \sup_X \Delta_{\omega}(\varphi) \leq C_1,
	\]
	where $C_1$ is a constant depending on $C_0$, and an upper bound of  $\|g\|_{C^2(X\times [-C_0,C_0])} $. 
\end{prop}
\begin{proof}
	The $L^{\infty}$-bound of $\varphi$ follows from the classical maximum principle. 
	
	Indeed  if 
 $M := \max \{\vert g(x,t) ; \, (x,t) \in X \times [-C_0,+C_0]\}$ and  $K > 0$ is large enough so that $ e^{-K+M} \omega_X^n \leq \omega^n \leq  e^{K- M} \omega_X^n$, then $- K$ is a subsolution and $+K$ is a supersolution of the equation \eqref{eq:MAKahler}. 
	
	To bound the Laplacian of $\varphi$ we introduce the function
	\[
	H:= \log {\rm Tr}_{\omega}(\omega_{\varphi}) + A_1(v+C_0 +1)^2 -A_2 \varphi, 
	\]
	where $A_1,A_2$ are positive constants to be determined later. Since $H$ is smooth on $X$ compact, its maximum is attained at some $x_0\in X$. 
	By Yau's inequality (see \cite{Yau78}, \cite{Siu87} \cite{GZbook}), we have 
	\[
	\Delta_{\omega_{\varphi}} \log {\rm Tr}_{\omega}(\omega_{\varphi}) \geq \frac{{\rm Tr}_{\omega}(-{\rm Ric}(\omega_{\varphi}))}{{\rm Tr}_{\omega}(\omega_{\varphi})} - B {\rm Tr}_{\omega_{\varphi}}(\omega),
	\]
	where $B$ is a lower bound for the holomorphic bisectional curvature of $\omega$. From the Monge-Amp\`ere equation $\omega_{\varphi}^n= e^{\varphi+ g(x,v)}\omega^n$, we obtain 
	\[
	{\rm Ric}(\omega_{\varphi}) = {\rm Ric}(\omega) - dd^c \varphi - dd^c (g(\cdot,v)). 
	\]
	Let $A_3,A_4>0$ be constants such that 
	\[
	{\rm Ric}(\omega) \leq A_4 \omega \; \text{and}\; \|g\|_{C^2(X\times [-C_0,C_0])} \leq A_3. 
	\] 
	By the Cauchy-Schwarz inequality we thus have 
	\[
	\Delta_{\omega_{\varphi}} \log {\rm Tr}_{\omega}(\omega_{\varphi}) \geq \frac{-A_5+ {\rm Tr}_{\omega}(\omega_{\varphi})-2A_3 |\partial v|^2_{\omega} -A_3 {\rm Tr}_{\omega}(\omega_v) }{{\rm Tr}_{\omega}(\omega_{\varphi})} - B {\rm Tr}_{\omega_{\varphi}}(\omega),
	\]
	where $A_5=A_4+A_3+n$. 
	We next compute 
	\begin{flalign*}
		\Delta_{\omega_{\varphi}} (v+C_0+1)^2 & = 2(v+C_0+1) {\rm Tr}_{\omega_{\varphi}}(\omega_v-\omega)+ 2 {\rm Tr}_{\omega_{\varphi}}(dv \wedge d^c v)\\
		& \geq -(4C_0+2) {\rm Tr}_{\omega_{\varphi}}(\omega) + 2{\rm Tr}_{\omega_{\varphi}}(\omega_v) + 2 {\rm Tr}_{\omega_{\varphi}}(dv \wedge d^c v),
	\end{flalign*}
	and 
	\[
	{\Delta}_{\omega_{\varphi}}(-A_2 \varphi) = -A_2n + A_2  {\rm Tr}_{\omega_{\varphi}}(\omega). 
	\]
	From Lemma \ref{lem: elementary inequalities} we have  
	\[
	{\rm Tr}_{\omega_{\varphi}}(\omega_v) 	{\rm Tr}_{\omega}(\omega_{\varphi}) \geq 	n{\rm Tr}_{\omega}(\omega_v),
	\]
	 \[ 
	 {\rm Tr}_{\omega_{\varphi}}(dv \wedge d^c v) 	{\rm Tr}_{\omega}(\omega_{\varphi}) \geq 	{\rm Tr}_{\omega}(dv \wedge d^c v) = |\partial v|^2_{\omega},
	\]
	\[
	{\rm Tr}_{\omega}(\omega_{\varphi}) {\rm Tr}_{\omega_{\varphi}}(\omega) \geq n^2. 
	\]
	Summing up we obtain, at $x_0$, 
	\begin{flalign*}
		0 &\geq  \Delta_{\omega_{\varphi}}(H)  \\
	&	= \Delta_{\omega_{\varphi}} \log {\rm Tr}_{\omega}(\omega_{\varphi}) + A_1\Delta_{\omega_{\varphi}}(v+C_0+1)^2 - A_2 \Delta_{\omega_{\varphi}}(\varphi) \\
	& \geq  \frac{-A_5+ {\rm Tr}_{\omega}(\omega_{\varphi})-2A_3 |\partial v|^2_{\omega} -A_3 {\rm Tr}_{\omega}(\omega_v) }{{\rm Tr}_{\omega}(\omega_{\varphi})} - B {\rm Tr}_{\omega_{\varphi}}(\omega)\\
	& -A_1(4C_0+2) {\rm Tr}_{\omega_{\varphi}}(\omega) + 2A_1 {\rm Tr}_{\omega_{\varphi}}(\omega_v) + 2A_1 {\rm Tr}_{\omega_{\varphi}}(dv\wedge d^c v) \\
	& -nA_2 + A_2 {\rm Tr}_{\omega_{\varphi}}(\omega)\\ 
		& \geq   \frac{-2A_3 |\partial v|^2_{\omega}+ (2A_1-A_3) {\rm Tr}_{\omega}(\omega_v)  + 2A_1   {\rm Tr}_{\omega} (dv \wedge d^c v)}{{\rm Tr}_{\omega}(\omega_{\varphi})} \\
		&+ (A_2-B-(4C_0+2)A_1-A_5/n^2) {\rm Tr}_{\omega_{\varphi}}(\omega) -nA_2+1. 
	\end{flalign*}
	Now we choose $A_1=A_3$ and $A_2=1+B+ A_1(4C_0+2) +A_5/n^2$, so that 
	\[
	0\geq {\rm Tr}_{\omega_{\varphi}}(\omega)  + 1- nA_2. 
	\]
	From this and Lemma \ref{lem: elementary inequalities} we obtain, at $x_0$,  
\[
{\rm Tr}_{\omega}(\omega_{\varphi}) \leq n{\rm Tr}_{\omega_{\varphi}}(\omega)^{n-1} e^{\varphi+g} \leq nA_2^{n-1}e^{2A_0}.
\]
Thus, $H(x) \leq H(x_0) \leq \log (nA_2^{n-1}e^{2A_0}) + A_1(2C_0+1)^2 +A_2A_0=:A_7$, and therefore 
${\rm Tr}_{\omega}(\omega_{\varphi}) \leq e^{A_7+A_2A_0}=:A_8$. 
\end{proof}

\subsection{Bounded solutions are smooth} \label{sect: Kahler smooth}
We are now ready to prove the main theorem for the case when $\theta=\omega=\omega_X$ is K\"ahler and $\psi^{\pm}=0$. 

\begin{proof}[Proof of Theorem \ref{thm: bdd implies smooth intro} in the K\"ahler case]
Let us recall the setting. We assume that $\varphi$ is a bounded $\omega$-psh function solving 
\[
(\omega+dd^c \varphi)^n = e^{f(x,\varphi)} \omega^n,
\]
where $f$ is a smooth function on $X\times \mathbb R$. We want to prove that $\varphi$ is smooth on $X$. Using Demailly's approximation theorem, we can find a sequence of smooth $\omega$-psh functions $v_j$ decreasing pointwise to $\varphi$ on $X$. We can assume that $-C_0\leq v_j\leq C_0$ for all $j$. For each $j$ we solve the Monge-Amp\`ere equation (see \cite{Aub78,Yau78})
\[
(\omega+dd^c u_j)^n = e^{u_j + g(x,v_j)} \omega^n,
\]
where $g(x,t):= -t+ f(x,t)$ is a smooth function on $X\times \mathbb{R}$, and $u_j$ is smooth. It follows from Proposition \ref{prop: Kahler case} that the sequence $\left(\sup_X |u_j|+ \sup_X |\Delta_{\omega} u_j|\right)$ is uniformly bounded. Hence, up to extracting a subsequence, we can assume that the sequence $(u_j)$ converges uniformly to $u\in {\rm PSH}(X,\omega)$ solving the Monge-Amp\`ere equation
\[
(\omega+dd^c u)^n = e^{u+ g(x,\varphi)}\omega^n.
\]
Since $(\omega+dd^c \varphi)^n =e^{\varphi+ g(x,\varphi)}\omega^n$, the comparison principle (see Proposition \ref{prop: max, comp, dom principle}) yields $u=\varphi$. 

 We claim that the sequence $(\Delta_{\omega}v_j)$ is also uniformly bounded. To prove the claim we recall the approximating sequence of Demailly, see \cite[(3.1)]{Dem94}: $v_j= \rho_{1/j}(u)$, where 
\[
\rho_{\varepsilon}(u)(x):=\int_{\zeta \in T_x X} u({\rm exph}_x(\zeta))\chi\left ( \frac{|\zeta|^2}{\varepsilon^2}\right)d\lambda(\zeta).  
\] 
From $\Delta u_j \leq C$, we deduce $\Delta
\rho_{\varepsilon}(u_j) \leq C'$, and letting $j\to +\infty$, we obtain $\Delta\rho_{\varepsilon}(u) \leq C'$, proving the claim.  

 Now, it follows from \cite[Theorem 14.9]{GZbook} that there exists $\alpha \in ]0,1[$ such that the $C^{2,\alpha}$-norm of $u_j$ is also uniformly bounded, thus $u\in C^{2,\alpha'}(X)$, for any $\alpha'<\alpha$.  
By Schauder's estimates (see \cite[Proposition 14.12]{GZbook}), $u$ is in fact smooth on $X$.   
\end{proof}

\subsection{Degenerate K\"ahler settings}
 In this section we prove our main theorem for degenerate data on compact K\"ahler manifolds. Recall that  $\theta$ is a smooth closed real $(1,1)$-form that is big and semipositive. We fix $\rho$ with analytic singularities such that $\theta+dd^c \rho \geq 2\varepsilon_0\omega_X$ for a fixed constant $\varepsilon_0$. Let $D$ be an analytic subset of $X$ such that $\rho$ is smooth in $X\setminus D$.  We normalize $\rho$ so that $\sup_X \rho = -1$.  Let $\psi^{\pm}$ be $\theta$-psh functions which are smooth in an open set $\Omega$. Let us emphasize here that $\Omega$ is not necessarily a Zariski open set. 
 
 \begin{proof}[Proof of Theorem \ref{thm: bdd implies smooth intro} in the degenerate K\"ahler setting]
 
 	We assume that $\varphi$ is a bounded $\theta$-psh function solving \eqref{eq: CMAeq}, and we want to prove that  $\varphi$ is smooth in $\Omega\setminus D$. Using Demailly's regularization theorem, we let $v_j,\psi^+_j,\psi_j^-$ be sequences of $\theta_j$-psh functions decreasing to $\varphi,\psi^+,\psi^-$ (respectively), where $\theta_j=\theta+j^{-1}\omega_X$. We can assume that $\psi_j^{\pm} \leq 0$. For each $j$, we solve 
 	\begin{equation}\label{eq: MA approximant K deg}
 		 	(\theta+j^{-1}\omega_X +dd^c u_j)^n = e^{u_j + g(\cdot,v_j)+\psi_j^+-\psi_j^-} \omega_X^n, 
 	\end{equation}
  where we define $g(\cdot, t) :=-t+ f(\cdot,t)$. In the next lemmas we will prove that 
   \begin{enumerate}
  	\item there exists a uniform constant $A_1$ with 
 \[
  A_1 \geq u_j \geq \varepsilon \psi_j^--C(\varepsilon), \; \text{for all}\;\varepsilon>0;
  \] 
  	\item there exists a uniform constant $A_2$ with 
  	\[
  	\Delta_{\omega_X}(u_j) \leq e^{-A_2(\psi^-+\rho )}. 
  	\]
  \end{enumerate}
  Thus, up to extracting we can assume $u_j\to u \in {\rm PSH}(X,\theta)$ in $L^1(X)$. We claim that $u\geq \varphi$. To prove the claim we consider, for each $b>1$, $w_{b,j}:=P_{\theta_j}(bu_j-(b-1)\psi_j^-)$. It follows from Lemma \ref{lem: b trick} and Lemma \ref{lem: envelope two forms} that 
  \begin{flalign*}
  	  (\theta +dd^c w_{b,j})^n &\leq b^n e^{b^{-1} w_{b,j} + g(\cdot,v_j) + \psi_j^{+}-b^{-1}\psi_j^-} \omega_X^n\\
  	  &\leq   e^{b^{-1} w_{b,j} + g(\cdot,v_j) + \psi_j^{+}-\psi_j^-+n\log b} \omega_X^n
  \end{flalign*}
  We next define 
  \[
  \tilde{w}_{b,j} := P_{\theta}\left(\inf_{k\geq j} w_{b,j}\right). 
  \]
  By the $L^{\infty}$ estimate in (1) above with $\varepsilon =1-b^{-1}$, we know that $w_{b,j} \geq -C(b)$. With this observation in hand, $\tilde{w}_{b,j}$ is well-defined as a bounded $\theta$-psh function. Note that the constant $C(b)$ depends a priori on $b$ and it may blow up as $b\to 1$ but we will show that it is not the case.

  It follows from Proposition \ref{prop: min principle} that
  \[
 (\theta +dd^c  \tilde{w}_{b,j})^n \leq e^{b^{-1} \tilde{w}_{b,j} + \sup_{k\geq j}(g(\cdot,v_j) + \psi_j^{+}-\psi_j^-)+n\log b} \omega_X^n
  \]
The sequence $(\tilde{w}_{b,j})_j$ is increasing almost everywhere to some  $\tilde{w}_b\in {\rm PSH}(X,\theta)\cap L^{\infty}$. Letting $j\to +\infty$, we obtain 
 \[
 (\theta +dd^c  \tilde{w}_{b})^n \leq e^{b^{-1} \tilde{w}_{b} + g(\cdot,\varphi) + \psi^{+}-\psi^-+n\log b} \omega_X^n
  \]
Moreover, using the boundedness of $\varphi$, we can also write
 \[
 (\theta +dd^c  \varphi)^n \geq e^{b^{-1} \varphi + g(\cdot,\varphi) + \psi^{+}-\psi^--(1-b^{-1})\sup_X |\varphi|} \omega_X^n. 
  \]
  By the domination principle (see Proposition \ref{prop: max, comp, dom principle}) we have 
  \[
  \tilde{w}_b+nb\log b\geq \varphi -(b-1)\sup_X |\varphi|.
  \] 
  From the definition of $w_{b,j}$, we also have 
  \[
  bu_j-(b-1)\psi_j^- \geq w_{b,j} \geq \tilde{w}_{b,j}. 
  \]
  Letting $j\to +\infty$, we see that 
  \begin{flalign*}
  	 bu& \geq (b-1)\psi^-+ \tilde{w}_b \\
  	 & \geq (b-1)\psi^- +\varphi -nb\log b-(b-1)\sup_X |\varphi|. 
  	  \end{flalign*}
   Finally, letting $b\to 1$ we obtain $u\geq \varphi$. 
   
   From \eqref{eq: MA approximant K deg}, considering $(\sup_{k\geq j} u_k)^*$, using Proposition \ref{prop: max, comp, dom principle},  and letting $j\to +\infty$, we can show that 
   \[
   (\theta+dd^c u)^n \geq e^{u + g(\cdot,\varphi) + \psi^{+}-\psi^-} \omega_X^n. 
   \]
   It follows again from the domination principle Proposition \ref{prop: max, comp, dom principle} that $u\leq \varphi$, hence $u=\varphi$. 
   
From the Laplacian estimate in (2), we can argue as in Section \ref{sect: Kahler smooth}, using the Evans-Krylov theory, to prove that $\varphi = u$ is smooth in $\Omega\setminus D$. 
 \end{proof}
 
 We now give the proof of the a priori estimates used in the above proof. 
 
 \begin{lemma}\label{lem: relative Linfty estimate}
 	With the same notations as above, there exists a uniform constant $A_1 > 0$ such that  
for any $\varepsilon>0$ there exists a constant $C(\varepsilon) >0$ with 
 \[
  A_1 \geq u_j \geq \varepsilon \psi_j^--C(\varepsilon), \, \, \text{for all j }. 
  \] 
 \end{lemma}
 To avoid repetition we present a proof that is also valid in the Hermitian setting.  
  \begin{proof}
  	We let $h\in {\rm PSH}(X,\theta_1)\cap L^{\infty}(X)$ solves 
  	\[
  	(\theta_1+ dd^c h)^n =e^{h+\psi^+} \omega_X^n. 
  	\]
  	The existence of $h$ is well-known (see \cite{BBGZ13}, \cite{GZbook} for the K\"ahler case and \cite{Ng16AIM} for the Hermitian case). 
  	
  	Since $\psi_j^+-\psi_j^-+g(\cdot,v_j) \geq \psi^+-C$, for a uniform constant $C$, we infer from the comparison principle that $u_j\leq h+C\leq \sup_X h +C=:A_1$.  To estbalish the lower bound, we consider the envelope 
  \[
  w_{b,j}:= P_{\theta_j}(bu_j-(b-1)\psi_j^-), b>1. 
  \]
  It follows from Lemma \ref{lem: b trick} that 
  \[
  (\theta+dd^c w_{b,j})^n \leq b^n e^{b^{-1} w_{b,j} + g(\cdot,v_j) + \psi_j^+-b^{-1}\psi_j^-} \omega_X^n. 
  \]
  Since $\varphi$ is bounded, we have that $e^{\psi^+-\psi^-}\in L^1(X,\omega_X^n)$. From the definition of the sequence $\psi_j$ and the convexity of $\exp$, we have that $e^{\psi^+_j-\psi^-_j}\leq (e^{\psi^+-\psi^-})_j$, and the latter converges in $L^1$ to $e^{\psi^+-\psi^-}$. The sequence $(e^{g(\cdot,v_j) + \psi_j^+-b^{-1}\psi_j^-})_j$ is bounded in $L^{b}(X)$ since 
  \[
   e^{b (g(\cdot,v_j) +  \psi_j^+-b^{-1}\psi_j^-)} \leq e^{C + \psi^+_j-\psi^-_j}.
  \]
  It thus follows from \cite[Theorem 3.4]{GL23Crelle} that $w_{b,j}\geq -C(b)$.  From the definition of $w_{b,j}$, we obtain $bu_j\geq (b-1)\psi_j^- -C(b)$, and the estimate follows with $\varepsilon=1-b^{-1}$. 
  \end{proof}
 
 \begin{prop}
	Assume $\omega$ is K\"ahler, $\theta \leq \omega\leq \omega_X$, $v$ is a smooth $\theta$-psh function, $-C_0\leq v\leq C_0$, $\psi^{\pm}$ are smooth $\omega$-psh functions on $X$  and $g : X \times [-C_0,C_0] \rightarrow \mathbb{R}$ is a smooth function.  Assume $\varphi$ is a smooth $\omega$-psh function solving 
\[
(\omega+dd^c \varphi)^n = e^{\varphi+g(x,v) + \psi^+-\psi^-}\omega_X^n. 
\] 
Then, for a uniform constant $C_1$, we have 
\[
\sup_{X } \Delta_{\omega_X} \varphi  \leq e^{-C_1 (\psi^-+\rho-1)}. 
\]
 \end{prop}
  
 \begin{proof}
 	We introduce the following function
\[
H(x):= \log {\rm Tr}_{\omega_X}(\omega_{\varphi}) + A_1(v+C_0 +1)^2 -A_2 \Big(\varphi- \rho - \varepsilon_0\psi^-\Big), 
\]
where $A_1,A_2$ are positive constants to be determined later.  Since there exists $C_1$ such that $\varepsilon_0\psi^--C_1 \leq \varphi \leq C_1 \, \, \, \textnormal{in} \, X $, $H$ is a smooth function in $X\setminus D$  and its maximum is attained at some $x_0\in X\setminus D$. 	By Yau's inequality we have 
	\[
	\Delta_{\omega_{\varphi}} \log {\rm Tr}_{\omega_X}(\omega_{\varphi}) \geq \frac{{\rm Tr}_{\omega_X}(-{\rm Ric}(\omega_{\varphi}))}{{\rm Tr}_{\omega_X}(\omega_{\varphi})} - B {\rm Tr}_{\omega_{\varphi}}(\omega_X),
	\]
	where $B$ is a lower bound for the holomorphic bisectional curvature of $\omega_X$. From the Monge-Amp\`ere equation $\omega_{\varphi}^n= e^{\varphi+g(x,v)+ \psi^+-\psi^-}\omega_X^n$, we obtain 
\[
{\rm Ric}(\omega_{\varphi}) = {\rm Ric}(\omega_X) - dd^c \varphi - dd^c g(x,v)-dd^c \psi^+ + dd^c \psi^-.
\]
By the Cauchy-Schwarz inequality we thus have 
\begin{flalign*}
\Delta_{\omega_{\varphi}} \log {\rm Tr}_{\omega_X}(\omega_{\varphi}) &\geq \frac{-A_3+ {\rm Tr}_{\omega_X}(\omega_{\varphi})-A_3 |\partial  v|^2_{\omega_X} -A_3 {\rm Tr}_{\omega_X}(\omega_v)- {\rm Tr}_{\omega_X}(\omega_{\psi^-}) }{{\rm Tr}_{\omega_X}(\omega_{\varphi})} \\
&- B {\rm Tr}_{\omega_{\varphi}}(\omega_X),
\end{flalign*}
where $A_3$ depends on the $C^2$-norm of $g$. We next compute 
	\begin{flalign*}
		\Delta_{\omega_{\varphi}} (v+C_0+1)^2 & = 2(v+C_0+1) {\rm Tr}_{\omega_{\varphi}}(\omega_v-\omega)+ 2 {\rm Tr}_{\omega_{\varphi}}(dv \wedge d^c v)\\
		& \geq -(8C_0+4) {\rm Tr}_{\omega_{\varphi}}(\omega_X) + 2{\rm Tr}_{\omega_{\varphi}}(\omega_v) + 2 {\rm Tr}_{\omega_{\varphi}}(dv \wedge d^c v). 
	\end{flalign*}
    We have 
     \begin{flalign*}
      -A_2\Delta_{\omega_{\varphi}}\Big(\varphi-\rho-\varepsilon_0\psi^- \Big)&=-A_2{\rm Tr}_{\omega_{\varphi}} \Big( \omega_{\varphi}-\varepsilon_0\omega_X-\varepsilon_0\omega_{\psi^-} \Big )\\& \geq  -A_2n + A_2\varepsilon_0  {\rm Tr}_{\omega_{\varphi}}(\omega_{\psi^-}) +A_2\varepsilon_0 {\rm Tr}_{\omega_{\varphi}}(\omega_X). 
       \end{flalign*}
       Note also that 
 	\[
	{\rm Tr}_{\omega_{\varphi}}(\omega_v) 	{\rm Tr}_{\omega_X}(\omega_{\varphi}) \geq 	n{\rm Tr}_{\omega_X}(\omega_v),
	\]
	 \[ 
	 {\rm Tr}_{\omega_{\varphi}}(dv \wedge d^c v) 	{\rm Tr}_{\omega_X}(\omega_{\varphi}) \geq 	{\rm Tr}_{\omega_X}(dv \wedge d^c v) = |\partial v|^2_{\omega_X},
	\]
	\[
	{\rm Tr}_{\omega_{\varphi}}(\omega_{\psi^-}) {\rm Tr}_{\omega_X}(\omega_{\varphi}) \geq {\rm Tr}_{\omega_X}(\omega_{\psi^-}). 
	\]
	Summing up we obtain, at $x_0$, 
	\begin{flalign*}
		0 \geq \Delta_{\omega_{\varphi}}(H) & \geq \left(A_2\varepsilon_0-A_1(4C_0+2)-B-A_3\right)  {\rm Tr}_{\omega_{\varphi}}(\omega_X) -nA_2 \\
		&+ \frac{(2A_1-A_3)|\partial v|_{\omega_X}^2}{{\rm Tr}_{\omega_X}(\omega_{\varphi})}  + \left(2A_1-A_3\right) \frac{{\rm Tr}_{\omega}(\omega_v)}{{\rm Tr}_{\omega}(\omega_{\varphi})}\\
		&+\left(A_2\varepsilon_0-1\right) \frac{{\rm Tr}_{\omega_X}(\omega_{\psi^-})}{{\rm Tr}_{\omega_X}(\omega_{\varphi})}. 
	\end{flalign*}
	Now we choose $A_1=\frac{A_3}{2}$  so that 
	\begin{flalign*}
	0 
		& \geq \left(A_2\varepsilon_0-A_1(4C_0+2)-B-A_3\right)  {\rm Tr}_{\omega_{\varphi}}(\omega_X) -nA_2 \\
				&+\left(A_2\varepsilon_0 -1\right) \frac{{\rm Tr}_{\omega_X}(\omega_{\psi^-})}{{\rm Tr}_{\omega_X}(\omega_{\varphi})},		\end{flalign*}
	and hence, by choosing $A_2= \varepsilon_0^{-1} (A_1(4C_0+2)+B+A_3+2)\geq 2\varepsilon_0^{-1}$, we arrive at
	\[
	0\geq {\rm Tr}_{\omega_{\varphi}}(\omega_X)  - nA_2. 
	\]
		From this and Lemma \ref{lem: elementary inequalities} we obtain, at $x_0$,  
\[
{\rm Tr}_{\omega_X}(\omega_{\varphi}) \leq n{\rm Tr}_{\omega_{\varphi}}(\omega_X)^{n-1} e^{\varphi+g+\psi^+-\psi-} \leq n(nA_2)^{n-1}e^{A_0+C_1-\psi^-},
\]
where $A_0:=\sup_{X\times [-C_0,C_0]} |g|$. 
Thus
\begin{flalign*}
	H(x) &\leq H(x_0) \\
	&\leq \log (n^nA_2^{n-1}e^{A_0+C_1-\psi^-(x_0)}) + A_1(2C_0+1)^2 -A_2(\varphi-\varepsilon_0 \psi^-)(x_0)\\
	& \leq C_2 +\left(\frac{A_2\varepsilon_0}{2} -1\right)\psi^-(x_0). 
\end{flalign*}
and therefore, using $A_2\varepsilon_0 \geq 2$, we arrive at $H(x)\leq C_2$. Hence 
${\rm Tr}_{\omega_X}(\omega_{\varphi}) \leq e^{-A_2(\psi^- +\rho-1)}$. 
 \end{proof}
 
It is natural to expect that Theorem \ref{thm: bdd implies smooth intro} is still valid if we merely assume that the solution $\varphi$ is in the full mass class $\mathcal E(X,\omega)$ in the K\"ahler case, as defined in \cite{GZ07}.  
\begin{coro}
 Assume $\varphi\in \mathcal E(X,\omega)$ solves 
 \[
 (\omega+dd^c \varphi)^n = e^{f(\cdot,\varphi) + \psi^+ - \psi^-}\omega_X^n,
 \]
 where $f$ is smooth in $X$ and quasi-increasing in the second variable. If $e^{\psi^+-\psi^-}\in L^p(X,\omega_X^n)$ for some $p >1$, then $\varphi \in L^{\infty} (X)$ and smooth in $\Omega$. 
\end{coro}
Quasi-increasing here means that there exists a constant $A>0$ such that  $g(\cdot,t) := f(\cdot,t) + At$ is increasing in $t$. When $\psi^+=0$, the assumption that $e^{-\psi^-}\in L^1(X)$ already implies $e^{-\psi^-}\in L^p(X)$, for some $p>1$, as follows from the resolution of the openness conjecture (see \cite{Bern15}, \cite{GZ15}, \cite{Hiep14}). 
\begin{proof}
	We can rewrite the above equation as 
	\[
	(\omega+dd^c \varphi)^n = e^{-A \varphi +g(\cdot,\varphi) + \psi^+ - \psi^-}\omega_X^n. 
	\]
Then  
 \[
 e^{-A \varphi +g(\cdot,\varphi) + \psi^+ - \psi^-}\leq e^{-A\varphi+ \sup_Xg(x,\sup_X \varphi) +\psi^+ - \psi^-}.
 \] 
 Since $e^{\psi^+-\psi^-} \in L^p(X)$ with $p > 1$, applying H\"older inequality with the exponent $p' := \sqrt{p} > 1$ and its conjugate exponent $q'$, we obtain
$$
\int_X e^{- p' (A \varphi +\psi^+- \psi^-) } \omega_X^n \leq \left(\int_X e^{- p(\psi^+- \psi^-)} \omega_X^n\right)^{1 \slash p'} \left( \int_X e^{- A p' q' \varphi } \omega_X^n\right)^{1 \slash q'}.
$$
The second term in the right hand side is finite  by Skoda's theorem \cite{Sko72} since the Lelong numbers of  $\varphi \in \mathcal E(X,\omega)$ vanishes (see \cite{GZ07}). Hence $\varphi$ is bounded in $X$ by Ko{\l}odziej \cite{Kol98}. Therefore, our main result applies giving the smoothness of $\varphi$ in $\Omega$.
\end{proof}

 \begin{remark}
 However, the conclusion in the above corollary is not true if we merely assume that $f$ is smooth. 
	
 Consider a holomorphic section $s$ of a holomorphic line bundle $L$ over $X$ and $\varphi= -2\log (-\log |s|^2)$ (one can take $X=\mathbb CP^n$). We can normalize the hermitian metric  on $L$ such that $|s|<<1$. Then 
	\[
	(\omega+dd^c \varphi)^n=e^{f(\cdot, \varphi)}\omega_X^n,
	\]
	where $f(\cdot,t)=h(x)+ 2t+e^{-t/2}$ for some smooth function $h$. Observe that $f$ is not quasi-increasing since $df/dt=2-e^{-t}$. 
 \end{remark}

 \section{The Hermitian case}\label{sect: Herm case}
Fix a compact hermitian manifold $(X,\omega_X)$ of dimension $n$, and let $\theta$ be a semipositive, smooth, real $(1,1)$-form that is big. By big we mean, following the terminology of \cite{GL23Crelle}, there exists $\rho \in {\rm PSH}(X,\theta)$, $\sup_X \rho =-1$, with analytic singularities along an analytic set $D\subset X$ such that $\theta+dd^c \rho \geq 2\varepsilon_0 \omega_X$, for some $\varepsilon_0>0$ that will be fixed throughout this section. To ease the notation we also assume that $2\theta\leq \omega_X$, so that when approximating $\theta$, we have $\theta \leq \theta+j^{-1}\omega_X \leq \omega_X$, for $j\geq 2$. We are studying solutions to the following complex Monge-Amp\`ere equation
\begin{equation}\label{eq: CMAE big herm}
    (\theta+dd^c \varphi)^n = e^{f(\cdot,\varphi) +\psi} \omega_X^n,
\end{equation}
where $f: X\times \mathbb{R} \rightarrow \mathbb{R}$ is a smooth function, $\psi=\psi^+-\psi^-$, with $\psi^{\pm}\in {\rm PSH}(X,\theta)$ smooth in an open set $\Omega\subset X$. The goal of this section is to prove Theorem \ref{thm: bdd implies smooth intro} in this general setting.
The existence of a bounded solution implicitly implies that $e^{\psi^+-\psi^-}\in L^1(X,\omega_X^n)$. 
As in the K\"ahler case, we prove our theorem by solving an appropriate family of Monge-Amp\`ere equations, and establishing a priori estimates. For each $j\geq 2$, we solve 
\[
(\theta_j+ dd^c u_j)^n = e^{u_j+g(\cdot,v_j)+\psi^+_j-\psi^-_j}\omega_X^n,
\]  
where $\theta_j=\theta+j^{-1}\omega_X$, $g(\cdot,t)=-t+f(\cdot,t)$, $v_j,\psi^+_j$, $\psi^-_j$ are smooth $\theta_j$-psh and decrease to $\varphi$, $\psi^+$, $\psi^-$ respectively. In the next subsection we establish uniform estimates for $u_j$. For notational convenience we will write $\omega$ for $\theta_j$ and drop the subscript $j$. 

\subsection{A priori estimates}

\begin{prop}\label{prop: estimate Hermitian}
	Assume the following 
	\begin{enumerate}
		\item $\omega$ is a hermitian form satisfying $\theta \leq \omega\leq \omega_X$, 
		\item  $v$ is a smooth $\omega$-psh function with $-C_0\leq v\leq C_0$,
		\item $\psi^{\pm}$ are smooth $\omega$-psh functions on $X$ with $\int_X e^{\psi^+-\psi^-}\omega_X^n\leq C_1$, 
		\item $g : X \times [-C_0,C_0] \rightarrow \mathbb{R}$ is a smooth function and set $C_2=\|g\|_{C^2(X\times [-C_0,C_0])}$.
	\end{enumerate}
	 Let $\varphi$ be the unique smooth $\omega$-psh function solving 
\[
(\omega+dd^c \varphi)^n = e^{\varphi+g(x,v) + \psi^+-\psi^-}\omega_X^n. 
\] 
Then for any $ \varepsilon>0$ there exists a constant $C(\varepsilon)>0$ such that 
\[
\varphi \geq \varepsilon \psi^- - C(\varepsilon),
\]
and
\[
\sup_{X } \Delta_{\omega_X} \varphi  \leq e^{-C (\psi^-+\rho)},
\]
where $C$ depends on $C_0$, $C_1$, $C_2$, $\varepsilon_0$, $X,\omega_X$, $n$. 
\end{prop}

The existence of $\varphi$ follows from \cite{Cher87}. 
 
 Several computations in the proof below are similar to those in \cite[Theorem 4.2]{GL23Crelle}, \cite{TW10a}, \cite{To18},  but we repeat them for completeness. We use $C_j$ to denote various constants that only depend on the fixed data. 
 \begin{proof}
 
 The relative $L^{\infty}$-estimate can be proved in the same way as we have done in Lemma \ref{lem: relative Linfty estimate}. Let us now proceed to the Laplacian estimate. 
 Replacing $v$ with $v+C_0+1$, we can assume that $1\leq v\leq 2C_0+1$.

 	Consider
\[
H:=\log {\rm Tr}_{\omega_X}(\tilde{\omega}) +Av^2-\gamma(u), 
\]
where $\tilde{\omega}=\omega+ dd^c\varphi$, $u=\varphi-\rho-\varepsilon_0\psi^-+C_3$, $A>0$ is a constant and $\gamma$ is a smooth increasing function (with $\gamma(+\infty)=+\infty$) that will be specified later. From the $L^{\infty}$-estimate we know that $u\geq 1$ (if we choose $C_3$ sufficiently large), therefore $H$ is bounded from above on $X$. 
Since $H(D)=-\infty$,  the maximum of $H$ is attained at some point $x_0 \in X\setminus D$. Since we want to bound $H(x_0)$ from above, we can assume that ${\rm Tr}_{\omega_X}(\tilde{\omega})\geq 1$ at $x_0$. 

We use special coordinates at this point, as defined by Guan-Li in \cite{GL10}, so that
\[
g_{i\bar j}=\delta_{ij}, \, \, g_{i \bar i j}:=\frac{\partial g_{i\bar i}}{\partial z_j}=0 \, \, {\rm and} \, \, \tilde{g}_{i \bar j} \, \, {\rm is \, \, diagonal},
\]
where $g$ denotes the Riemannian metric associated to $\omega_X$ while $\tilde{g}$ corresponds to $\tilde{\omega}=\omega+dd^c \varphi$. To construct the desired coordinates we first use a linear transformation to make $\tilde{g}_{i\bar j}$ diagonal (at $x_0$), and then apply (2.19) in \cite{GL10}. 

The Laplacian of $H$ can be computed as
\[
\Delta_{\tilde{\omega}} H=\frac{\Delta_{\tilde{\omega}}{\rm Tr}_{\omega_X}(\tilde{\omega})}{{\rm Tr}_{\omega_X}(\tilde{\omega})}-\frac{|\partial {\rm Tr}_{\omega_X}(\tilde{\omega})|^2_{\tilde {\omega}}}{({\rm Tr}_{\omega_X}(\tilde{\omega}))^2} +A\Delta_{\tilde{\omega}}v^2- \gamma'(u)\Delta_{\tilde{\omega}}(u)-\gamma''(u)|\partial u|^2_{\tilde {\omega}}.
\]
We first take care of $\Delta_{\tilde{\omega}}{\rm Tr}_{\omega_X}(\tilde{\omega})$,
\begin{align*}
    \Delta_{\tilde{\omega}}{\rm Tr}_{\omega_X}(\tilde{\omega})&=\sum_{i,j,k,l}\tilde{g}^{i \bar j} (g^{k\bar l}\tilde{g}_{k \bar l})_{i\bar j}\\&=\sum_{i,k}\tilde{g}^{i \bar i}\tilde{g}_{k \bar k  i\bar i}-2 \Re \left( \sum_{i,j,k}\tilde{g}^{i \bar i}g_{j\bar k \bar i}\tilde{g}_{k \bar j i}\right)+  \sum_{i,k}\tilde{g}^{i \bar i}g_{j\bar k i}g_{j\bar k \bar i}\tilde{g}_{k \bar k }\\&+ \sum_{i,j,k}\tilde{g}^{i \bar i}g_{k\bar j i}g_{j\bar k \bar i}\tilde{g}_{k \bar k }-\sum_{i,j,k}\tilde{g}^{i \bar i}{g}_{k \bar k \bar i i}\tilde{g}_{k \bar k}\\&\geq
    \sum_{i,k}\tilde{g}^{i \bar i}\tilde{g}_{k \bar k  i\bar i}-2 \Re \left( \sum_{i,j,k}\tilde{g}^{i \bar i}g_{j\bar k \bar i}\tilde{g}_{k \bar j i}\right) -C_4{\rm Tr}_{\omega_X}(\tilde{\omega}){\rm Tr}_{\tilde{\omega}}(\omega_X).
\end{align*}
We also have 
\[
{\rm Tr}_{\omega_X}({\rm Ric (\tilde {\omega}}))=\sum_{i,k}\tilde{g}^{i \bar i}\left( -\tilde{g}_{ i \bar i k\bar k }+ \sum_{j}\tilde{g}^{j \bar j}\tilde{g}_{i\bar j k }\tilde{g}_{j\bar i \bar k}\right),
\]
where ${\rm Ric (\tilde {\omega}})$ is the Chern-Ricci form of $\tilde {\omega}$.

Using these inequalities and the fact that $|\tilde{g}_{ i \bar i k\bar k }-\tilde{g}_{  k\bar k i \bar i }| \leq C_5$, we obtain
\begin{flalign}
    \Delta_{\tilde{\omega}}{\rm Tr}_{\omega_X}(\tilde{\omega}) &\geq \sum_{i,j,k}\tilde{g}^{i \bar i}\tilde{g}^{j \bar j}\tilde{g}_{i\bar j k }\tilde{g}_{j\bar i \bar k} -{\rm Tr}_{\omega_X}({\rm Ric (\tilde {\omega}})) -C_6{\rm Tr}_{\tilde{\omega}}(\omega_X) \label{eq:LaplacianTrace}\\& -2 \Re \left( \sum_{i,j,k}\tilde{g}^{i \bar i}g_{j\bar k \bar i}\tilde{g}_{k \bar j i}\right) -C_7{\rm Tr}_{\omega_X}(\tilde{\omega}){\rm Tr}_{\tilde{\omega}}(\omega_X). \nonumber 
\end{flalign}
Since 
\[
\tilde \omega^n=e^{\varphi+ g(x,v)+\psi^+-\psi^-}\omega_X^n,
\]
we have
\[
{\rm Ric (\tilde{\omega}})={\rm Ric ( {\omega_X}})-dd^c(\varphi + g(x,v)+\psi^+-\psi^-).
\]
By the Cauchy-Schwarz inequality we thus have
\begin{align*}
-{\rm Tr}_{\omega_X}({\rm Ric (\tilde {\omega}}))&=-{\rm Tr}_{\omega_X}({\rm Ric ({\omega_X}}))+{\rm Tr}_{\omega_X}(\tilde{\omega}-\omega)\\
& +{\rm Tr}_{\omega_X}(dd^c g(x,v)) 
+{\rm Tr}_{\omega_X}(\omega_{\psi^+}-\omega_{\psi^-})\\
&\geq
-C_8 -C_9|\partial v|^2_{\omega_X}- C_{10}{\rm Tr}_{\omega_X}(\omega_v)-{\rm Tr}_{\omega_X}(\omega_{\psi^-}),
\end{align*}
and therefore,
\begin{flalign}
    \Delta_{\tilde{\omega}}{\rm Tr}_{\omega_X}(\tilde{\omega}) &\geq \sum_{i,j,k}\tilde{g}^{i \bar i}\tilde{g}^{j \bar j}\tilde{g}_{i\bar j k }\tilde{g}_{j\bar i \bar k} -C_8-C_9|\partial v|^2_{\omega_X}- C_{10}{\rm Tr}_{\omega_X}( {\omega}_v)-{\rm Tr}_{\omega_X}(\omega_{\psi^-})\label{eq:LaplacianTrace2}\\
    & -C_6{\rm Tr}_{\tilde{\omega}}(\omega_X)  -2 \Re \left( \sum_{i,j,k}\tilde{g}^{i \bar i}g_{j\bar k \bar i}\tilde{g}_{k \bar j i}\right) -C_7{\rm Tr}_{\omega_X}(\tilde{\omega}){\rm Tr}_{\tilde{\omega}}(\omega_X). \nonumber 
\end{flalign}
Since the special choice of coordinates at $x_0$ ensures that $g_{j\bar j \bar i}=0$, using Cauchy-Schwarz inequality and $|\tilde{g}_{ k \bar j i }-\tilde{g}_{  i\bar j k }| \leq C_5$, we obtain
\begin{align*}
    \left| 2 \Re \left( \sum_{i,j,k}\tilde{g}^{i \bar i}g_{j\bar k \bar i}\tilde{g}_{k \bar j i}\right)\right| &\leq \left|  2 \Re \left( \sum_i\sum_{j\neq k}\tilde{g}^{i \bar i}g_{j\bar k \bar i}\tilde{g}_{i \bar j k}\right) \right|+ C_{11}{\rm Tr}_{\tilde{\omega}}(\omega_X) \\ & \leq \sum_i\sum_{j\neq k}\Big(\tilde{g}^{i \bar i}\tilde{g}^{j \bar j}\tilde g_{i\bar j k}\tilde{g}_{j \bar i \bar k} + \tilde{g}^{i \bar i}\tilde{g}_{j \bar j} g_{ j \bar k \bar i }{g}_{k \bar j i}\Big) + C_{11}{\rm Tr}_{\tilde{\omega}}(\omega_X) \\ & \leq \sum_i\sum_{j\neq k} \tilde{g}^{i \bar i}\tilde{g}^{j \bar j}\tilde g_{i\bar j k }\tilde{g}_{j \bar i \bar k} + C_{12}{\rm Tr}_{{\omega_X}}(\tilde\omega){\rm Tr}_{\tilde{\omega}}(\omega_X)+ C_{11}{\rm Tr}_{\tilde{\omega}}(\omega_X).
\end{align*}
Together with \eqref{eq:LaplacianTrace2} this yields
\begin{flalign}\label{eq:LaplacianTrace3}
	 \Delta_{\tilde{\omega}}{\rm Tr}_{\omega_X}(\tilde{\omega}) & \geq I -C_8-C_9|\partial v|^2_{\omega_X}- C_{10}{\rm Tr}_{\omega_X}( {\omega}_v) -{\rm Tr}_{\omega_X}(\omega_{\psi^-}) \\
	 &
	 -C_{13}{\rm Tr}_{{\omega_X}}(\tilde\omega){\rm Tr}_{\tilde{\omega}}(\omega_X)- C_{14}{\rm Tr}_{\tilde{\omega}}(\omega_X), \nonumber
\end{flalign}
where $I:=\sum_{i,j} \tilde{g}^{i \bar i}\tilde{g}^{j \bar j}\tilde g_{i\bar j j }\tilde{g}_{j \bar i \bar j}$. We next compute
\begin{align*}
    |\partial {\rm Tr}_{\omega_X}(\tilde{\omega})|^2_{\tilde {\omega}}&= \sum_{i,j,k} \tilde{g}^{i \bar i}\tilde g_{j\bar j i }\tilde{g}_{k  \bar k \bar i}\\&=
    \sum_{i,j,k} \tilde{g}^{i \bar i}(T_{ij\bar j}+ \tilde g_{i\bar j j })(\overline{ T}_{ik\bar k}+\tilde{g}_{k   \bar i\bar k })\\&=
    \sum_{i,j,k} \tilde{g}^{i \bar i}\tilde g_{i\bar j j }\tilde{g}_{k   \bar i\bar k }+\sum_{i,j,k} \tilde{g}^{i \bar i}T_{ij\bar j}\overline{ T}_{ik\bar k} +2\Re\left(\sum_{i,j,k} \tilde{g}^{i \bar i}T_{ij\bar j}\tilde{g}_{k   \bar i\bar k }\right),    
\end{align*}
where $T_{ij\bar j}=\tilde g_{j\bar j i }- \tilde g_{i\bar j j }$ is under control, $|T_{ij\bar j}|\leq C_5$. By Cauchy-Schwarz inequality, we obtain
\begin{align*}
    \sum_{i,j,k} \tilde{g}^{i \bar i}\tilde g_{i\bar j j }\tilde{g}_{k   \bar i\bar k }&=\sum_{i} \tilde{g}^{i \bar i}\left|\sum_{j}\tilde g_{i\bar j j }\right|^2\\&\leq
    \left( \sum_{i,j} \tilde{g}^{i \bar i}\tilde{g}^{j \bar j}\tilde g_{i\bar j j }\tilde{g}_{j \bar i\bar j }\right)\left( \sum_{j} \tilde{g}_{j \bar j}\right)=I{\rm Tr}_{\omega_X}(\tilde{\omega}).
\end{align*}
We thus get
\begin{equation}\label{eq:PartialTrace}
    \frac{|\partial {\rm Tr}_{\omega_X}(\tilde{\omega})|^2_{\tilde {\omega}}}{({\rm Tr}_{\omega_X}(\tilde{\omega}))^2} \leq \frac{I}{{\rm Tr}_{\omega_X}(\tilde{\omega})}+C_{15} \frac{{\rm Tr}_{\tilde{\omega}}(\omega_X)}{({\rm Tr}_{\omega_X}(\tilde{\omega}))^2}+ \frac{2}{({\rm Tr}_{\omega_X}(\tilde{\omega}))^2}\Re\left(\sum_{i,j,k} \tilde{g}^{i \bar i}T_{ij\bar j}\tilde{g}_{k   \bar i\bar k }\right).
\end{equation}
We have
\[
({\rm Tr}_{\omega_X}(\tilde{\omega}))_{\bar i}=\sum_{k,l} \Big(-{g}^{k\bar p}{g}^{q\bar l} g_{p\bar q \bar i }\tilde{g}_{k \bar l } + g^{k\bar l }\tilde{g}_{k \bar l \bar i } \Big).
\]
Since $H_{\bar i}=0$ and $ {g}_{k \bar k \bar i }=0 $ at $x_0$, we obtain by differentiating $H$ once 
\[
\sum_k\tilde{g}_{k \bar k \bar i }={\rm Tr}_{\omega_X}(\tilde{\omega})\Big(2Avv_{\bar i}-\gamma'(u)u_{\bar i}\Big).
\]
Cauchy-Schwarz inequality yields
\begin{align*}
&\left| \frac{2}{({\rm Tr}_{\omega_X}(\tilde{\omega}))^2}\Re\left(\sum_{i,j,k} \tilde{g}^{i \bar i}T_{ij\bar j}\tilde{g}_{k   \bar k \bar i }\right) \right|\\
 &\leq C_{16}\frac{\gamma'(u)^2}{(-\gamma''(u))}\frac{{\rm Tr}_{\tilde{\omega}}(\omega_X)}{({\rm Tr}_{\omega_X}(\tilde{\omega}))^2} + (-\gamma''(u))|\partial u |^2_{\tilde \omega}
+\frac{A|\partial v |^2_{\tilde \omega}+ AC_{17}{\rm Tr}_{\tilde{\omega}}(\omega_X) }{{\rm Tr}_{\omega_X}(\tilde{\omega})}.
\end{align*}
Since $|\tilde {g}_{k \bar k \bar i }-\tilde {g}_{k \bar i\bar k  }| \leq C_5$, we have
\begin{align*}
\left| \frac{2}{({\rm Tr}_{\omega_X}(\tilde{\omega}))^2}\Re\left(\sum_{i,j,k} \tilde{g}^{i \bar i}T_{ij\bar j}\tilde{g}_{k   \bar i \bar k  }\right) \right| &\leq C_{16}\left(\frac{\gamma'(u)^2}{(-\gamma''(u))}+ 1 \right)\frac{{\rm Tr}_{\tilde{\omega}}(\omega_X)}{({\rm Tr}_{\omega_X}(\tilde{\omega}))^2} + (-\gamma''(u))|\partial u |^2_{\tilde \omega}\\&
+A\frac{|\partial v |^2_{\tilde \omega}}{{\rm Tr}_{\omega_X}(\tilde{\omega})}+AC_{17}\frac{{\rm Tr}_{\tilde{\omega}}(\omega_X)}{{\rm Tr}_{\omega_X}(\tilde{\omega})}.
\end{align*}

We also have
\[
A \Delta_{\tilde \omega}v^2 \geq 2A{\rm Tr}_{\tilde{\omega}}(\omega_v)+2A {\rm Tr}_{\tilde{\omega}}(dv \wedge d^cv)-A(4C_0+2){\rm Tr}_{\tilde{\omega}}(\omega_X).
\]
Since $0\geq \Delta_{\tilde \omega} H$ at $x_0$, it follows from this, \eqref{eq:LaplacianTrace3}, and \eqref{eq:PartialTrace}
that 
\begin{flalign}
    0 & \geq \Delta_{\tilde \omega} H  \\& \nonumber \geq \frac{ -C_8 -C_9|\partial v|^2_{\omega_X}- C_{10}{\rm Tr}_{\omega_X}( {\omega}_v) -C_{13}{\rm Tr}_{\omega_X}(\tilde\omega){\rm Tr}_{\tilde{\omega}}(\omega_X)- C_{14}{\rm Tr}_{\tilde{\omega}}(\omega_X)-{\rm Tr}_{\omega_X}(\omega_{\psi^-})}{{\rm Tr}_{\omega_X}(\tilde {\omega})}\\& \nonumber -C_{16}\left(\frac{\gamma'(u)^2}{(-\gamma''(u))}+ 2 \right)\frac{{\rm Tr}_{\tilde{\omega}}(\omega_X)}{({\rm Tr}_{\omega_X}(\tilde{\omega}))^2} - (-\gamma''(u))|\partial u |^2_{\tilde \omega}
-A\frac{|\partial v |^2_{\tilde \omega}}{{\rm Tr}_{\omega_X}(\tilde{\omega})}\\& \nonumber-AC_{17}\frac{{\rm Tr}_{\tilde{\omega}}(\omega_X)}{{\rm Tr}_{\omega_X}(\tilde{\omega})}+ 2A{\rm Tr}_{\tilde{\omega}}(\omega_v)+2A {\rm Tr}_{\tilde{\omega}}(dv \wedge d^cv)-A(4C_0+2){\rm Tr}_{\tilde{\omega}}(\omega_X)\\& \nonumber - \gamma'(u){\rm Tr}_{\tilde{\omega}}(dd^c u)-\gamma''(u)|\partial u|^2_{\tilde {\omega}}.
\end{flalign}

Note that $dd^c u=\omega_{\varphi}- \omega_{\rho} - \varepsilon_0 dd^c \psi^-\leq \tilde{\omega}- \varepsilon_0\omega_X-\varepsilon_0\omega_{\psi^-}$. 
For positive $(1,1)$-forms $\alpha$, $\beta$, $\gamma$, we have ${\rm Tr}_{\alpha}(\beta){\rm Tr}_{\beta}(\gamma)\geq {\rm Tr}_{\alpha}(\gamma)$, therefore,  
\begin{flalign*}
	{\rm Tr}_{\tilde{\omega}}(dv \wedge d^cv){\rm Tr}_{\omega_X}(\tilde{\omega} )&\geq {\rm Tr}_{\omega_X}(dv \wedge d^cv)=|\partial v |_{\omega_X}^2,\\
	{\rm Tr}_{\tilde{\omega}}(\omega_v){\rm Tr}_{\omega_X}(\tilde{\omega})
	& \geq {\rm Tr}_{\omega_X}(\omega_v), \\
	{\rm Tr}_{\tilde{\omega}}(\omega_{\psi^-}){\rm Tr}_{\omega_X}(\tilde{\omega}) &\geq {\rm Tr}_{\omega_X}(\omega_{\psi^-}). 
\end{flalign*}
Thus taking $A=C_9+C_{10}$, we obtain
\begin{flalign}\label{eq:ineqality-max}
0 &\geq (\varepsilon_0\gamma'(u)-C_{18}){\rm Tr}_{\tilde{\omega}}(\omega_X)-C_{16}\left(\frac{\gamma'(u)^2}{(-\gamma''(u))}+ 2 \right)\frac{{\rm Tr}_{\tilde{\omega}}(\omega_X)}{({\rm Tr}_{\omega_X}(\tilde{\omega}))^2} \\& \nonumber-C_{19}\frac{{\rm Tr}_{\tilde{\omega}}(\omega_X)}{{\rm Tr}_{\omega_X}(\tilde{\omega})}- n\gamma'(u)+ \frac{(\varepsilon_0 \gamma'(u)-1) {\rm Tr}_{\omega_X}(\omega_{\psi^-})}{{\rm Tr}_{\omega_X}(\tilde{\omega})}.
\end{flalign}

We now choose the function $\gamma$ as
\[
\gamma(u):=\frac{2(C_{18}+1)}{\min(\varepsilon_0,1)}u+\log u.
\]
Since $u\geq 1$ we have 
\[
\gamma'(u) \geq C_{20}+1, \; -\gamma''(u)\leq 1 \; \textnormal{and} \, \, \frac{\gamma'(u)^2}{(-\gamma''(u))}+ 2 \leq C_{20}u^2.
\]

By incorporating this into \eqref{eq:ineqality-max} we obtain 
\begin{flalign}\label{eq:ineqality-max_2}
    0&\geq {\rm Tr}_{\tilde{\omega}}(\omega_X)-C_{21}u^2\frac{{\rm Tr}_{\tilde{\omega}}(\omega_X)}{({\rm Tr}_{\omega_X}(\tilde{\omega}))^2}-C_{21}\frac{{\rm Tr}_{\tilde{\omega}}(\omega_X)}{{\rm Tr}_{\omega_X}(\tilde{\omega})}- C_{21}.
\end{flalign}

At the point $x_0$ there are two possibilities. 
\begin{enumerate}
	\item If ${\rm Tr}_{\omega_X}(\tilde{\omega})\geq 4C_{21}u$ then from \eqref{eq:ineqality-max_2} we get ${\rm Tr}_{\tilde{\omega}}(\omega_X) \leq C_{22}$. Using
    $$
    {\rm Tr}_{\omega_X}(\tilde \omega)\leq n\frac{\tilde{\omega}^n}{\omega_X^n}({\rm Tr}_{\tilde{\omega}}(\omega_X))^{n-1}\leq C_{23}e^{-\psi^-}
    $$
    and $u-\rho -\varepsilon_0 \psi^- \geq -\frac{\varepsilon_0}{2}\psi^- -C_{24}$ we then get $H(x_0)\leq C_{25}$. 
    \item    
If ${\rm Tr}_{\omega_X}(\tilde{\omega})\leq 4C_{21}u$ then from the definition of $H$, we get 
\[
H(x_0) \leq \log u +C_{26} -\gamma(u) \leq C_{26}.
\] 
\end{enumerate}

 In both cases we have $H(x_0)\leq C_{28}$, hence $H(x) \leq C_{29}$ and the estimate follows. 
 \end{proof}

\subsection{Proof of Theorem \ref{thm: bdd implies smooth intro}}
We are finally ready to prove Theorem \ref{thm: bdd implies smooth intro} in the Hermitian case. 
Recall that $\varphi \in PSH(X,\theta) \cap L^{\infty}(X)$ is a solution to the equation
$$
(\theta + dd^c \varphi)^n = e^{f (x,\varphi)+\psi^+ - \psi^-}\omega_X^n,
$$
where $\psi^{\pm}$ is smooth in $\Omega$.
The goal is to prove that $\varphi$ is smooth in $\Omega$.
\begin{proof}
   For each $j>0$ we set $\theta_j=\theta+j^{-1}\omega_X$. We approximate $\varphi$ from above by using Demailly's approximation theorem: there exists a decreasing sequence $v_j\in {\rm PSH}(X,\theta_j)\cap C^{\infty}(X)$  converging pointwise to $\varphi$. We can assume that $-C_0\leq v_j\leq C_0$. 
   
   We write  $f(x,t) = t + g(x,t)$, where  $g(x,t) := f(x,t) -t$ is a smooth function.  By Aubin-Yau theorem (\cite{Aub78,Yau78}), for each $j$ there exists $u_j \in PSH(X,\theta_j) \cap C^{\infty} (X)$ solving the equation
   \[
   (\theta_j +dd^c u_j)^n = e^{u_j+g(x,v_j)+\psi_j}\omega_X^n,
   \]
where $\psi_j=\psi_j^+-\psi_j^-$, and $\psi_j^{\pm}$ are the smooth aproximants of $\psi^{\pm}$ using Demailly's convolution.
Thus $\psi^{\pm}_j \in {\rm PSH}(X,\theta_j)$ for each $j$ and, by Jensen's inequality  $e^{\psi_j^+-\psi_j^-} \leq f_j,$
 where $f := e^{\psi^+ - \psi^-}$ and $(f_j)$ is the smooth aproximants of $f$ using Demailly's convolution. Since $ f_j \to f $ in $L^1(X))$, it follows follows from the Lebesgue dominated convergence theorem that (up to taking a subsequence), $e^{\psi_j^+-\psi_j^-} \to e^{\psi^+-\psi^-}$ in $L^1(X)$ as $j \to + \infty$. We refer to \cite{DnL17} and \cite[Proof of Theorem 4.2]{GL23Crelle} for the details. 
 
 By the estimate in Proposition \ref{prop: estimate Hermitian} above we have 
\[
-C_1 + \frac{\psi_j^-}{2} \leq u_j\leq C_1, \; |\Delta_{\omega_X}(u_j)|\leq e^{-C_2 \psi_j^--C_2 \rho}\leq e^{-C_2 \psi^{-}-C_2\rho}.
\]
Up to extracting a subsequence, we can assume that $u_j\to u\in {\rm PSH}(X,\theta)$ in $L^1(X,\omega_X^n)$. Now, arguing as in the K\"ahler case we can show that $u=\varphi$ is smooth in $\Omega\setminus D$. 
\end{proof}

As a direct consequence of our a priori estimates above we obtain the folowing regularity result. 
\begin{proposition}
	Assume the same setting as in Theorem \ref{thm: bdd implies smooth intro}. If $\varphi \in PSH(X,\theta)\cap L^{\infty}(X)$ solves 
	\[
	(\theta+dd^c \varphi)^n = e^{f(x,v) + \psi^+-\psi^-} \omega_X^n, 
	\]
	for some $v\in PSH(X,\omega_X)\cap L^{\infty}(X)$. Then $\Delta \varphi$ is locally bounded in $\Omega \setminus D$. 
\end{proposition}
\begin{proof}
	We approximate $v$ from above by a sequence of smooth $\omega_X$-psh functions $v_j$ and solve 
	\[
	(\theta+2^{-j} \omega_X +dd^c \varphi_j)^n = e^{\varphi_j-v_j+f(x,v_j)+\psi^+_j-\psi_j^-} \omega_X^n. 
	\]
	Our a priori estimates ensure that $\Delta \varphi_j$ is locally uniformly bounded in $\Omega\setminus D$, hence so is $\Delta \varphi$. 
\end{proof}

\section{The subsolution theorem} \label{sect: subsol thm}
The goal of this section is to prove that if the equation \eqref{eq:CMAeq2} admits a bounded subsolution lying below a bounded supersolution then it admits a bounded solution lying between them.

Recall our setup. We consider a  compact Hermitian manifold  $X$ of dimension $n$ and fix a Hermitian metric $\omega_X$ on $X$. Let $\theta$ be a smooth semipositive $(1,1)$-form that is big. By big, we mean there exists a function $\rho$ with analytic singularities along an analytic subset $D$ such that $\theta +dd^c \rho \geq 2\varepsilon_0 \omega_X$, for some $\varepsilon_0>0$.

We consider the following complex Monge-Amp\`ere equation 
 \begin{equation} 
 \label{eq:CMAeq2}
 (\theta+dd^c u)^n = e^{f(\cdot,u) + \psi^+-\psi^{-}} \omega_X^n,
 \end{equation}
where $f: X \times \mathbb R \rightarrow \mathbb R$ is a smooth function, and $\psi^{\pm}$ are quasi-psh functions on $X$ which are smooth in a fixed  open set $\Omega\subset X$ (not necessarily Zariski dense in $X$). 

We assume moreover that
\begin{equation} \label{eq:Lpcondition}
\int_X e^{ p \psi } \omega_X^n < + \infty,
\end{equation}
for some $p >1$, where $\psi := \psi^+-\psi^{-}$.

When  $t \longmapsto f(\cdot,t) $ is non increasing, the equation does not necessarily admit a solution even if the data are smooth and if a solution exists it is not unique in general. Nevertheless, our next result shows that, if the equation admits a bounded subsolution that lies below a bounded supersolution, then the equation admits a bounded solution which is smooth in the regular locus of the data.
 
 \begin{theorem} \label{thm:ExistenceBoundedSol}  Assume that the equation \eqref{eq:CMAeq2}  admits a subsolution $\underline u \in PSH(X,\theta) \cap L^{\infty}(X)$ and a  supersolution $\bar u \in PSH(X,\theta) \cap L^{\infty}(X)$ such that $\underline u  \leq \bar u$ in $X$. 
 
 Then the following properties hold :
 
 \noindent 1)  the equation   \eqref{eq:CMAeq2} admits  a bounded solution $u \in PSH(X,\theta) \cap L^{\infty}(X)$ which is smooth in $\Omega \setminus D$ and satisfies $\underline u \leq u \leq \bar u$ in $X$;
 
 \noindent 2)  the maximal subsolution 
 $$
 U := \sup \{v \in PSH(X,\theta) \cap L^{\infty}(X) ; \underline u \leq v \leq \bar u, (\theta + dd^c v)^n \geq e^{f(\cdot,v) + \psi^+ - \psi^-} \omega_X^n\}
 $$ 
 is a bounded solution of the equation \eqref{eq:CMAeq2} which is smooth in $\Omega\setminus D$;
 
 \noindent 3) the minimal supersolution $V := P_\theta(\bar w)$, where
 $$\bar w :=  \inf\{w \in PSH(X,\theta) \cap L^{\infty}(X) ; \underline u \leq w \leq \bar u, (\theta + dd^c w)^n \leq e^{f(\cdot,w)+\psi^+ - \psi^-} \omega_X^n\}
 $$
 is a bounded solution  of the equation \eqref{eq:CMAeq2} which is smooth in $\Omega\setminus D$.
 \end{theorem}

When $\theta$ is Hermitian (i.e. $\theta>0$), any bounded $\theta$-psh solution to \eqref{eq: CMAE big herm} is automatically  continuous on $X$ since the right-hand side belongs to $L^p(X,dV_X)$; see Remark 3.13 in \cite{KN19}. When $\theta=\pi^*\omega_V$ is the pullback of some Hermitian form on some singular complex variety $V$ with klt singularities, the continuity of $u$ has been recently proved in \cite{Dang26}, building on \cite{CC25}.

 \subsection{Existence of a bounded solution}
 To prove the main result of this section we first establish the existence of a bounded solution given the existence of sub/super solutions.  We consider a more general degenerate complex Monge-Ampère equation
  \begin{equation}  \label{eq:CMAeq3}
 (\theta+dd^c u)^n = e^{f(\cdot,u) } g\omega_X^n,
 \end{equation}
 where $0 \leq g \in L^p(X)$ with $p > 1$ and  $\int_X g\omega_X^n > 0$.

 \begin{defi}  \noindent 1. We say that $\underline u \in {\rm PSH}(X,\theta) \cap L^{\infty}(X)$ is a subsolution of the equation \eqref{eq:CMAeq3} if it satisfies the inequality 
 $$
 (\theta+dd^c \underline u)^n  \geq e^{f(\cdot,\underline u)} g\omega_X^n, 
 $$
 in the weak  sense on $X$.
 
 \noindent 2. We say that $\bar u \in PSH(X,\theta) \cap L^{\infty}(X)$ is a supersolution of the equation \eqref{eq:CMAeq3} if it satisfies the inequality 
 $$
 (\theta+dd^c \bar u)^n  \leq e^{f(\cdot,\bar u)} g\omega_X^n, 
 $$
 in the weak  sense on $X$.
 \end{defi} 

 
 \begin{proposition} \label{prop:Existence} Assume that the  equation \eqref{eq:CMAeq3}
 admits a subsolution $\underline u  \in {\rm PSH}(X,\theta) \cap L^{\infty}(X) $ and a supersolution $\bar u \in {\rm PSH}(X,\theta) \cap L^{\infty}(X) $  such that $\underline u \leq \bar u$. Then it admits  a solution $u \in {\rm PSH}(X,\theta) \cap L^{\infty}(X) $ which  satisfies  $\underline u \leq u \leq \bar u$ on $X$.
  \end{proposition}
   \begin{proof} 
  We use the iterative method as in \cite{BZ23}. Let $C>0$ be a constant such that $-C\leq \underline{u}\leq \overline{u}\leq C$. Since $f$ is smooth, we can fix a constant $A>0$ such that $\partial_t f \leq A$ in $X\times [-C,C]$. It follows that the function $h(x,t):= f(x,t)-At$ is non-increasing in $t$.  
  
   We set $u_0=\underline{u}$ and define by induction  a sequence $(u_k)$ of bounded $\theta$-plurisubharmonic functions such that for each $k \in \N$,
  \[
  (\theta+dd^c u_k)^n = e^{Au_k +h(\cdot,u_{k-1})} g\omega_X^n.
  \]
   This equation admits a unique solution, as follows from \cite[Theorem B]{GL23Crelle} (the case when $\theta>0$ was proved in \cite{Ng16AIM}). Observe that  
   \[
   (\theta+dd^c u_0)^n \geq e^{Au_0+h(\cdot,u_0)}g\omega_X^n,
   \]
   and, using the fact that $h$ is non-increasing in $t\in [-C,C]$ and $u_0\leq \overline{u}$, 
   \[
   (\theta+dd^c \overline{u})^n \leq e^{A\overline{u}+ h(\cdot,u_0)}g\omega_X^n.
   \]
   The comparison principle (Proposition \ref{prop: max, comp, dom principle}) thus ensures that $u_1\geq u_0$. Since $h$ is non-increasing in $t$, we infer that 
   \[
   (\theta+dd^c u_1)^n \geq e^{Au_1+h(\cdot,u_1)} g \, \omega_X^n,
   \]
    We can thus repeat the above argument to see that $u_2\geq u_1$, and so on.  Therefore, $\underline{u}\leq u_k \leq u_{k +1} \leq \bar u$ for every $k\in \N$. Hence the sequence $(u_k)$ is non-decreasing in $k$ and it converges almost everywhere to a function  $u_{\infty} \in PSH(X,\theta) \cap L^{\infty}(X)$. By Bedford-Taylor's convergence theorem $u_{\infty}$ is  a solution of the equation \eqref{eq:CMAeq3}. 
 \end{proof}
 
 \subsection{The non-decreasing case} 
We next consider a particular case when $f$ is non-decreasing in $t$. For each $t$ fixed, as shown in \cite[Theorem B]{GL23Crelle}, there exists a unique $c(t) \in \mathbb R$, and $\varphi_t\in {\rm PSH}(X,\theta)$ such that 
\begin{equation}
	\label{eq: sol increasing}
	(\theta+dd^c \varphi_t)^n = e^{c(t) + f(\cdot,t)} g \, \omega_X^n. 
\end{equation}
The goal of this subsection is to give a necessary and sufficient condition for the existence of a solution to the equation \eqref{eq:CMAeq3}
in terms of the existence of a special solution to the equation \eqref{eq: sol increasing}.

\begin{theorem}\label{thm: bdd sol f increasing}
	If $f$ is non-decreasing in $t$ then the following conditions are equivalent: 
	\begin{enumerate}
		\item there exists $t_0\in \mathbb R$ such that $c(t_0)=0$; 
		\item there exists a bounded subsolution and a bounded supersolution of \eqref{eq:CMAeq3};
		\item there exists a bounded solution of \eqref{eq:CMAeq3}.
	\end{enumerate}
\end{theorem}

Note that in the condition (2) above we do not assume any ordering relation between the subsolution and the supersolution, but it holds automatically after adding constants.

We now give a simple example where the condition (1) above is satisfied. Assume $f$ satisfies 
\[
\lim_{t\to \pm\infty}f(x,t)=\pm \infty,
\]
uniformly in $x$ (one can take e.g. $f(.,t)=e^t$). Let $u$ be a bounded $\theta$-psh function and $c(g)$ be a constant such that 
\[
(\theta+dd^c u)^n = e^{c(g)} gdV_X.
\] 
Then, for $t_1$ near $-\infty$, we have $f(.,t_1) \leq c(g)$ and for $t_2$ near $+\infty$, we have $f(.,t_2) \geq c(g)$. It then follows from the domination principle that $c(t_1) \geq 0$ and $c(t_2) \leq 0$. The continuity of $t\mapsto c(t)$ (see \cite[Corollary 1.2]{LPT21}) then ensures the existence of $t_0$ with $c(t_0)=0$.

 In \cite{Lu13}, the K\"ahler case (when $\theta=\omega_X$ is K\"ahler) was studied and the condition (1) can be reformulated as $\int_X e^{f(\cdot,t_0)}\omega_X^n=\int_X \theta^n$. Our result above generalizes the one in \cite{Lu13} to the general Hermitian context, where the argument using computation of Monge-Amp\`ere mass fails. Our proof  is quite different, relying on the comparison principle and the resolution of the equation given sub/super solutions in Theorem \ref{thm:SubSolThm}.
\begin{proof}
The implication $(3)\Longrightarrow (2)$ is obvious.  

We prove $(2)\Longrightarrow (1)$. Assume there are a bounded subsolution $\underline{u}$ and a bounded supersolution $\bar u$, then for $t_1=\inf_X \underline{u}$, $t_2=\sup_X \bar u$, the domination principle (see Proposition \ref{prop: max, comp, dom principle}) yields $c(t_1)\geq 0$ and $c(t_2)\leq 0$. Indeed we have
\begin{eqnarray*}
(\theta + dd^c \varphi_{t_1})^n &= & e^{c(t_1) + f(\cdot,t_1)} g \omega_X^n \\
&\leq & e^{c(t_1) + f(\cdot,\underline{u})} g \omega_X^n \\
&\leq & e^{c(t_1)} (\theta + dd^c \underline{u})^n.
\end{eqnarray*}
The domination principle, Proposition \ref{prop: max, comp, dom principle}, then yields $c(t_1) \geq 0$.  
By the same reasoning we see that $c(t_2) \leq 0$.

 The continuity of $t\mapsto c(t)$ (see \cite[Corollary 1.2]{LPT21}) then ensures the existence of $t_0$ such that $c(t_0)=0$.

We finally prove the implication $(1)\Longrightarrow (3)$. Let $u := \varphi_{t_0} \in PSH(X,\theta)\cap L^{\infty}$ be a solution to \eqref{eq: sol increasing} with $t=t_0$. Then $\underline{u}:=u-\sup_X u +t_0$ is a bounded subsolution of \eqref{eq:CMAeq3}, while $\overline{u}:=u+\inf_X u +t_0$ is a bounded supersolution. Since $\underline{u}\leq \overline{u}$, Proposition \ref{prop:Existence} implies the existence of a bounded solution of \eqref{eq:CMAeq3} lying between $\underline{u}$ and $\overline{u}$.
\smallskip
\end{proof}

%

\begin{remark}
	 However, the existence of a subsolution alone does not ensure the existence of a solution. One can take $f$ strictly increasing but $f\leq -1$, $\psi^{\pm}=0$, and $\theta=\omega_X$. Indeed, if $u$ is a bounded solution then $(\omega_X+dd^c u)^n\leq e^{-1}\omega_X$. The domination principle Proposition \ref{prop: max, comp, dom principle} then ensures that $u\geq C$ for all $C>0$, which is impossible. 
\end{remark}

\begin{remark}
Since $f$ is non-decreasing in $t$, the domination principle Proposition \ref{prop: max, comp, dom principle} reveals that $c(t)$ is non-increasing in $t$.  Indeed if $t_1<t_2$ then
$$(\theta+dd^c \varphi_{t_1})^n \leq e^{c(t_1) - c(t_2)} (\theta+dd^c \varphi_{t_2})^n.$$
The domination principle, Proposition \ref{prop: max, comp, dom principle} then yields $c(t_1)\geq c(t_2)$.

We do not assume here that $f$ is increasing in $t$, therefore, $c(t)$ may be constant in $t$. 

\end{remark}

 \subsection{From weak to smooth solution}
We now prove Theorem \ref{thm:ExistenceBoundedSol}
 \begin{proof} The existence of a solution $u \in PSH(X,\theta) \cap L^{\infty} (X)$ of the equation \eqref{eq:CMAeq2} such that $\underline u \leq u \leq \bar u$ follows from Proposition \ref{prop:Existence} since $g := e^{\psi^+ - \psi^-} \in L^p(X)$ with $p > 1$.
 
 By the maximum principle $V \in PSH(X,\theta) \cap L^{\infty} (X)$ is a subsolution of the equation \eqref{eq:CMAeq2} such that $\underline u \leq V \leq \bar u$.
 By Proposition \ref{prop:Existence}, there exists a solution $u \in PSH(X,\theta) \cap L^{\infty} (X)$ of the equation \eqref{eq:CMAeq2} such that $ \underline u \leq V \leq u \leq \bar u$. By definition of $V$ we also have $ u \leq V$. Hence $V =u$ is a solution of the equation \eqref{eq:CMAeq2}.
 
 Since $\underline u \leq W$, it follows that $W \in PSH(X,\theta) \cap L^{\infty} (X) $. By the minimum principle it follows that $W$ is a supersolution of the equation \eqref{eq:CMAeq2} (see \cite{GLZ19}). We again apply Proposition \ref{prop:Existence}  to produce a solution $u \in PSH(X,\theta) \cap L^{\infty} (X)$  of the equation \eqref{eq:CMAeq2} such that $\underline u \leq u \leq W $.
 By definition we have $ W \leq u$, hence  $W = u$ is a solution of the equation \eqref{eq:CMAeq2}.
 
 In each case the smoothness follows from Theorem \ref{thm: bdd implies smooth intro}.
  \end{proof}

\begin{example}\label{example: non-increasing f}
	Consider a smooth function $f : X\times \mathbb R \rightarrow \mathbb R$ such that $\lim_{t\to -\infty} f(x,t) <0$ and $\lim_{t\to +\infty}f(x,t) >0$ uniformly in $x$. For sufficiently large $C>0$, the constant functions $-C$, and $C$ are sub/super solutions of the equation 
	\[
	(\omega_X+dd^c u)^n =e^{f(\cdot,u)} \omega_X^n. 
	\]
	The above result yields a bounded solution lying between $-C$ and $C$. By our regularity result, this solution is smooth on $X$.  
\end{example}  

\begin{example} \label{example:general f}
	Take $f(x,t)=-t+h(x)$, with $f$ smooth and assume that the equation 
	\[
	(\omega+dd^c u)^n =e^{f(\cdot,u)} \omega^n
	\]
	does not have any bounded solution. One can construct such an example on a Fano manifold without K\"ahler-Einstein metrics. Then, the constant function $\underline{u}=C$ is a subsolution when $C$ is large enough, while $-C$ will be a supersolution. Thus, the existence of sub/super solutions (but not with the right order) does not ensure the existence of a solution. 
\end{example}

\end{document}